\documentclass[12pt]{amsart}
\usepackage{epsfig}
\usepackage{amssymb, amsmath}
\usepackage{amsfonts,graphics,epsfig}
\usepackage{mathrsfs}
\usepackage{pb-diagram, pb-xy}
\usepackage[all]{xy}
\usepackage{color}
\usepackage{tikz}
\usepackage{setspace}
\usepackage[left=1.5in,right=1.5in,top=1.5in,bottom=1.5in]{geometry}
\usepackage{color}   
\usepackage{hyperref}
\hypersetup{
    colorlinks=true, 
    linktoc=all,     
    linkcolor=red,  
	citecolor=blue
}

\newcommand{\mbR}{\mathbb{R}}

\newcommand{\mbZ}{\mathbb{Z}}
\newcommand{\mbQ}{\mathbb{Q}}
\def\mbF{\mathbb{F}}

\def\mbP{\mathbb{P}}

\newcommand{\<}{\leq}
\def\>{\geq}

\def\vphi{\varphi}

\newcommand{\lrd}{\lfloor}
\newcommand{\rrd}{\rfloor}

\newcommand{\num}{\equiv}

\def\mcO{\mathcal{O}}

\def\msH{\mathscr{H}}

\def\msM{\mathscr{M}}

\def\injective{\hookrightarrow}

\newcommand{\rtmap}{\dashrightarrow}

\newtheorem{theorem}{Theorem}[section]
\newtheorem{lemma}[theorem]{Lemma}
\newtheorem{proposition}[theorem]{Proposition}
\newtheorem{corollary}[theorem]{Corollary}

\theoremstyle{remark}
\newtheorem{remark}[theorem]{Remark}
\theoremstyle{definition}
\newtheorem{definition}[theorem]{Definition}
\theoremstyle{definition}

\numberwithin{equation}{section}

\def\div{\operatorname{div}}

\def\dim{\operatorname{dim}}
\def\codim{\operatorname{codim}}
\def\chr{\operatorname{char}}

\def\Spec{\operatorname{Spec}}

\author{Omprokash Das}
\address{School of Mathematics\\
Tata Institute of Fundamental Research\\
Homi Bhabha Road, Navy Nagar\\
Colaba, Mumbai 400005}
\email{omprokash@gmail.com, omdas@math.tifr.res.in}

\date{}

\begin{document}
\title[Kawamata-Viehweg Vanishing Theorem]{Kawamata-Viehweg Vanishing Theorem for del Pezzo Surfaces over imperfect fields of characteristic $p>3$}

\thanks{The author was partially supported by the AMS-Simons Travel Grant Award.}
\maketitle
\begin{abstract}
	In this article we prove that the Kawamata-Viehweg vanishing theorem holds for regular del Pezzo surfaces over imperfect ground fields of characteristic $p>3$.
\end{abstract}
\tableofcontents

\section{Introduction}
It is well-known that the Kodaira vanishing fails in positive characteristic. First counterexample was constructed by Raynaud on a smooth projective surface over algebraically closed field of every positive characteristic \cite{Ray78}. Further such counterexamples were studied in \cite{DI87, Eke88}, \cite[Section 2.6]{Kol96} and \cite{Mad16, Muk13}. A stronger version of Kodaira vanishing in characteristic $0$ is called the Kawamata-Viehweg vanishing theorem. It is known that the Kawamata-Viehweg vanishing theorem fails for Fano varieties in positive characteristic. In \cite{LR97} Lauritzen and Rao constructed a counterexample for Fano varieties of dimension at least $6$ defined over an algebraically closed field of characteristic $2$ for which Kodaira vanishing theorem fails. Recently Totaro \cite{Tot17} showed that Kodaira vanishing fails for Fano varieties in every positive characteristic $p>0$.  Over imperfect field of characteristic $2$, Schr\"oer \cite{Sch07} and Maddock \cite{Mad16} constructed a regular (but not smooth) del Pezzo surface $X$ such that $h^1(X, \mcO_X)\neq 0$, violating the Kodaira vanishing theorem. In the same paper Maddock also asked the question whether this kind of example i.e., a regular del Pezzo surface $X$ with $h^1(X, \mcO_X)\neq 0$ exists in characteristic $p>2$ \cite[Question 5.1]{Mad16}. He speculated that such examples would not exist in characteristic $p>3$. It has been confirmed recently by Patakfalvi and Waldron \cite[Theorem 1.9]{PW17} that indeed Kodaira vanishing holds for regular del Pezzo surfaces over imperfect fields in characteristic $p>3$. On the other hand, Cascini and Tanaka showed in \cite{CT16} that the Kawamata-Viehweg vanishing theorem holds for smooth del Pezzo surfaces over algebraically closed field of arbitrary positive characteristic. In the same paper they also showed that the same vanishing theorem fails for smooth rational surfaces over algebraically closed field in every positive characteristic. Another related result along this line is by Cascini, Tanaka and Witaszek, in \cite{CTW17} they showed that the Kawamata-Viehweg vanishing theorem holds for KLT log Fano surfaces over algebraically closed field in high enough characteristic. Recently Wang and Xie showed in \cite{WX17} that the Kawamata-Viehweg vanishing theorem holds for toric surfaces over arbitrary field of characteristic $p>0$.\\
In light of the recent developments in vanishing theorems for Fano varieties in positive characteristic (see \cite[Theorem 1.9]{PW17}) it is natural to ask whether the stronger version of Kodaira vanishing, namely the Kawamata-Viehweg vanishing theorem still holds for regular del Pezzo surfaces over \emph{imperfect} fields. We answer this question affirmatively in characteristic $p>3$. In particular, we prove the following theorem.
\begin{theorem}[Theorem \ref{thm:main}]
Let $(X, \Delta\>0)$ be a projective KLT pair of dimension $2$ over an arbitrary field $k$ of characteristic $p>3$. Let $D$ be a $\mbZ$-divisor on $X$ such that $D-(K_X+\Delta)$ is nef and big. Further assume that one of the following conditions is satisfied:
\begin{enumerate}
	\item $X$ is a regular del Pezzo surface i.e., $X$ is regular and $-K_X$ is ample, or
	\item $X$ is a normal del Pezzo surface, i.e., $X$ is normal and $-K_X$ is an ample Cartier divisor, and $D\>0$ is an effective $\mbZ$-divisor.
\end{enumerate}
 Then $H^i(X, \mcO_X(D))=0$ for all $i>0$.
	\end{theorem}
Varieties over imperfect fields appear naturally in positive characteristic, even when we work only over algebraically closed fields. For example, if $f:X\to Z$ is a morphism between two varieties over an algebraically closed field $k$ of characteristic $p>0$, then the generic fiber $X_{\eta}$ of $f$ is a variety over the function field $K(Z)$ of $Z$, which is an imperfect field. On the other hand, if $X$ is a variety over an imperfect field $K$ which is finitely generated over an algebraically closed field $k$, then there exist two varieties $\mathcal{X}$ and $\mathcal{B}$ over $k$ and a morphism $f:\mathcal{X}\to\mathcal{B}$ such that the function filed of $\mathcal{B}$ is $K$ and $X$ is the generic fiber of $f$. Another motivation for studying regular varieties over imperfect fields is the following: let $f:X\to Y$ be a morphism between two varieties over algebraically closed field of positive characteristic, and $X$ is smooth. Then it is not true in general that the general fibers of $f$ are smooth varieties, counterexamples are known to exist, for example, quasi-elliptic fibration in characteristic $2$ and $3$. However, the generic fiber $X_\eta$ of $f$ is a regular variety (the local rings of $X_\eta$ are all regular local rings) and thus lots of well-known results, for example, MMP, Abundance etc. still hold on the generic fiber $X_\eta$ when $\dim X_\eta=2$ (see \cite{Tan18}).\\       
	   
	   We were informed by Fabio Bernasconi that he found a counterexample for the Kawamata-Viehweg vanishing theorem for KLT log del Pezzo surfaces over algebraically closed field $k$ of characteristic $p=3$ \cite{Ber17}.\\   
	   
{\bf Acknowledgements.} I learned about this problem in a conversation with Professor Paolo Cascini at the `Conference in Birational Geometry' at the Simons Foundation in New York (August 21-25, 2017). I am grateful to Professor Cascini for our discussion. I would also like to thank the organizers of the conference and Simons Foundation for their hospitality during the conference. My special thanks go to Professor Burt Totaro, Joe Waldron and Fabio Bernasconi for carefully reading an early draft and pointing out some mistakes.\\

\section{Preliminaries}
We will work over an arbitrary field $k$ (possibly \emph{imperfect}) in characteristic $p>0$, unless stated otherwise. A \emph{variety} $X$ is an integral separated scheme of finite type over a field $k$. A \emph{surface} is a variety of dimension $2$. A variety $X$ is called \emph{regular} if the local rings $\mcO_{X, x}$ are all regular local rings for all closed points $x\in X$. For the definitions of MMP singularities see \cite[Chapter 2]{KM98, Kol13}. By \emph{$\mbZ$-divisor} $D$ we mean an integral Weil divisor.\\

\begin{definition}\label{def:del-pezzo}
A projective surface $X$ is called a \emph{regular del Pezzo} surface if $X$ is regular and $-K_X$ is an ample divisor. $X$ is called a \emph{normal del Pezzo} surface if $X$ is normal and $-K_X$ is an ample Cartier divisor. A \emph{Mori dream space} is a $\mbQ$-factorial normal projective variety $X$ such that every $D$-MMP terminates for every $\mbQ$-divisor $D$ on $X$. \\
\end{definition}

Note that our definition of \emph{Mori dream space} is quite general and it is not same as what is traditionally been called the Mori dream space in \cite{HK00}. We work in this generality for the convenience of our proof. 

\begin{remark}\label{rmk:smoothness-vs-regularity}
	Note that over a perfect base field $k$ (e.g., algebraically closed fields, finite fields $\mbF_{p^e}$ for any $e>0$, etc.), a variety $X$ is regular if and only if it is smooth over $k$. However, if $k$ is not perfect, then these two concepts do not coincide. In general, if $X$ is smooth over $k$, then $X$ is regular.\\ 
\end{remark}

\begin{definition}
For a real number $r\in\mbR$, we define $\lrd r\rrd$ to be the \emph{largest integer} less than or equal to $r$. For a divisor $D=\sum r_iD_i$, where $r_i\in\mbR$, we define $\lrd D\rrd=\sum\lrd r_i\rrd D_i$.\\
\end{definition}

\begin{remark}\label{rmk:round-down}
 Note that, if $r\>0$, then $\lrd r\rrd\>0$. For any $n\in\mbZ$ and $r\in\mbR$, $\lrd n+r\rrd=n+\lrd r\rrd$. Furthermore, $\lrd \cdot \rrd$ is a monotonically increasing function on $\mbR$, i.e., for any two real numbers $a$ and $b$ with $a\< b$, $\lrd a\rrd\<\lrd b\rrd$ holds. We also know that, $\lrd s+t\rrd\>\lrd s\rrd+\lrd t\rrd$ for two arbitrary real numbers $s, t\in \mbR$.\\		
\end{remark}

\subsection{Serre duality for $\mbZ$-divisors on normal surface}\label{sbs:serre-duality} A normal surface $X$ is always Cohen-Macaulay, since it is $R_1$ and $S_2$. For a $\mbZ$-divisor $D$ on $X$, the associated divisorial sheaf $\mcO_X(D)$ is reflexive and thus torsion free, i.e., $S_1$,  and $S_2$; in particular $\mcO_X(D)$ is a Cohen-Macaulay sheaf. Therefore by \cite[Proposition 5.75 and Theorem 5.71]{KM98} Serre duality holds and we have 
\[H^i(X, \mcO_X(D))\cong H^{2-i}(X, \msH om(\mcO_X(D), \mcO_X(K_X)))^*,\]
for all $i\>0$. Let $U$ be the regular locus of $X$ and $\iota:U\injective X$ the open immersion. Since $X$ is normal, $\codim_X (X-U)\>2$. Then on $U$ we have $\msH om(\mcO_X(D), \mcO_X(K_X))|_U\cong\mcO_X(K_X-D)|_U$. Thus $\msH om(\mcO_X(D), \mcO_X(K_X))\cong\iota_*(\msH om(\mcO_X(D), \mcO_X(K_X))|_U)\cong \iota_*(\mcO_X(K_X-D)|_U)\cong \mcO_X(K_X-D)$, since all the sheaves involved here are reflexive and $\codim_X(X-U)\>2$. In particular, the Serre duality takes the following standard form:
\[H^i(X, \mcO_X(D))\cong H^{2-i}(X, \mcO_X(K_X-D))^*,\]
for all $i\>0$.\\

\begin{remark}\label{rmk:gorenstein}
	A scheme $X$ is Gorenstein if and only if $X$ is Cohen-Macaulay and the canonical sheaf $\omega_X$ is a line bundle (see \cite[Definition 2.58, Page 79]{Kol13}). Therefore a normal surface $X$ with $K_X$ a Cartier divisor is a Gorenstein surface. In particular, from Definition \ref{def:del-pezzo} we see that all del Pezzo surfaces are Gorenstein.\\
\end{remark}

\subsection{MMP singularities over perfect field and base change}\label{sbs:perfect-base-change} Let $(X, \Delta\>0)$ be a pair over a perfect ground field $k$. Let $\bar{k}$ be the algebraic closure of $k$. Then the extension $\bar{k}/k$ is separable, since $k$ is perfect. Let $X_{\bar{k}}$ be the base change to the algebraic closure and $\Delta_{\bar{k}}$ is the flat pullback of $\Delta$ to $X_{\bar{k}}$. We note that $X_{\bar{k}}$ is normal in this case but may have disjoint irreducible components (see \cite[Remark 2.7(1)]{GNT15}); however, if $H^0(X, \mcO_X)=k$, then $X_{\bar{k}}$ is irreducible (see \cite[Lemma 2.2]{Tan15}).\\
There exists a finite sub-extension $k'$ of $\bar{k}/k$ such that the coefficients of the local equations of $X_{\bar{k}}$ and $\Delta_{\bar{k}}$ on an open cover of $X_{\bar{k}}$ are all contained in the field $k'$; in particular, the pair $(X_{\bar{k}}, \Delta_{\bar{k}})$ is defined over the field $k'$, and thus the singularities of $(X_{\bar{k}}, \Delta_{\bar{k}})$ are same as those of $(X_{k'}, \Delta_{k'})$. Since $k'/k$ is a finite separable algebraic extension, the base-change morphism $X_{k'}\to X$ is smooth. Then by \cite[Proposition 2.15]{Kol13}, $(X_{k'}, \Delta_{k'})$ is KLT (resp. DLT, PLT, LC, terminal or canonical) if $(X, \Delta)$ is KLT (resp. DLT, PLT, LC, terminal or canonical). In particular, $(X_{\bar{k}}, \Delta_{\bar{k}})$ is KLT (resp. DLT, PLT, LC, terminal or canonical) if $(X, \Delta)$ is KLT (resp. DLT, PLT, LC, terminal or canonical).\\

\subsection{$\mbZ$-divisors on KLT surfaces}\label{sbs:klt-surface} If $(X, \Delta\>0)$ is a KLT surface pair over an arbitrary field $k$, then $X$ is $\mbQ$-factorial by \cite[Corollary 4.11]{Tan18}. So we can work freely with $\mbZ$-divisors on $X$.

\section{Lemmas and Propositions}
In this section we prove some lemmas and propositions which will be used in the next section in the proof of the main theorem. The first three lemmas and their corollaries are known to the experts, however, we could not find proper references, so we prove them here.
\begin{lemma}\label{lem:positive-base-change}
	Let $X$ be a geometrically normal proper variety defined over an arbitrary field $k$. Let $k'/k$ be a field extension and $X_{k'}=X\times_k\Spec k'$ the corresponding base change. If a $\mbQ$-Cartier $\mbQ$-divisor $D$ on $X$ is ample (resp. nef, resp. big) then its flat-pullback $D_{k'}$ to $X_{k'}$ is also ample (resp. nef, resp. big).	
\end{lemma}
\begin{proof}
Replacing $D$ by a sufficiently divisible multiple we may assume that $D$ is Cartier. Then $mD$ is very ample for $m\gg 0$, since $D$ is ample. Therefore it gives an embedding of $X$ into a projective space: $\phi_{|mD|}:X\injective \mbP H^0(X, \mcO_X(mD))^*$. Base changing this morphism to $k'$ and noticing the fact that $H^0(X_{k'}, \mcO_{X_{k'}}(D_{k'}))=H^0(X, \mcO_X(D))\otimes_k k'$, we see that $mD_{k'}$ gives an embedding of $X_{k'}$: $\phi_{|mD_{k'}|}:X_{k'}\injective H^0(X_{k'}, \mcO_{X_{k'}}(mD_{k'}))$ for all $m\gg 0$. Therefore $D_{k'}$ is an ample divisor on $X_{k'}$.\\
	
	If $D$ is nef, then for an ample Cartier divisor $A$ on $X$, $mD+A$ is ample for all $m\>0$. Then by the previous result, $mD_{k'}+A_{k'}$ is ample on $X_{k'}$ for all $m\>0$. Thus $D_{k'}+\frac{1}{m}A_{k'}$ is an ample $\mbQ$-Cartier divisor on $X_{k'}$ for all $m>0$. Therefore by taking limit as $m\to+\infty$ we get that $D_{k'}$ a nef Cartier divisor.\\
	
	If $D$ is big, then the associated rational map $\phi_{|mD|}:X\rtmap \mbP H^0(X, \mcO_X(mD))^*$ is birational to its image for $m\gg 0$. Then by a similar argument as in the ample case we see that $D_{k'}$ is big on $X_{k'}$.\\
	
\end{proof}

\begin{lemma}[Projection Formula]\label{lem:projection-formula}
	Let $f:X\to Y$ be a proper birational morphism between two normal varieties over an arbitrary field $k$. Let $D$ be a $\mbQ$-Cartier $\mbZ$-divisor on $Y$. Then for any $f$-exceptional effective $\mbZ$-divisor $E\>0$ the following holds:
	\[
		f_*\mcO_X(\lrd f^*D\rrd+E)\cong \mcO_Y(D).
	\]	
\end{lemma}

\begin{proof}
	Since the question is local on the base, we may assume that $Y$ is affine. In particular, it is enough to show that $H^0(X, \mcO_X(\lrd f^*D\rrd+E))\cong H^0(Y, \mcO_Y(D))$. Note that $f^*:K(Y)\to K(X)$ is an isomorphism of function fields, since $f$ is birational.\\
	Now let $\vphi\in H^0(Y, \mcO_Y(D))$. Then $D+\div(\vphi)\>0$, and thus $f^*D+\div(f^*\vphi)\>0$. Then we have $\lrd f^*D\rrd+\div(f^*\vphi)\>0$ (see Remark \ref{rmk:round-down}), and hence $\lrd f^*D\rrd+E+\div(f^*\vphi)\>0$, since $E\>0$. In particular, $f^*\vphi\in H^0(Y, \mcO_X(\lrd f^*D\rrd+E))$. On the other hand, if $f^*\psi\in \mcO_X(\lrd f^*D\rrd+E))$, then $\lrd f^*D\rrd+E+\div(f^*\psi)\>0$; in particular, $f^*D+E+\div(f^*\psi)\>0$. Thus by applying $f_*$ we get, $D+\div(\psi)\>0$, i.e., $\psi\in H^0(Y, \mcO_Y(D))$. Therefore $H^0(X, \mcO_X(\lrd f^*D\rrd+E))\cong H^0(Y, \mcO_Y(D))$.\\ 
\end{proof}

\begin{corollary}\label{cor:projection-formula}
Let $f:X\to Y$ be a proper birational morphism between two normal varieties over an arbitrary field $k$. Let $D$ be a $\mbQ$-Cartier $\mbZ$-divisor on $X$ such that $f_*D$ is also a $\mbQ$-Cartier divisor on $Y$. Write $D=f^*f_*D+E$.\\
If $E$ is an effective $f$-exceptional $\mbQ$-divisor on $X$, then $f_*\mcO_X(D)\cong \mcO_Y(f_*D)$. 	
\end{corollary}

\begin{proof}
	We have $D=\lrd D\rrd=\lrd f^*f_*D+E\rrd\>\lrd f^*f_*D\rrd+\lrd E\rrd$. Since $f_*D$ is a $\mbZ$-divisor on $Y$, the difference $\left(\lrd f^*f_*D+E\rrd\right)-\left(\lrd f^*f_*D\rrd+\lrd E\rrd\right)$ is an effective $f$-exceptional $\mbZ$-divisor on $X$. Thus we may write $D=\lrd f^*f_*D\rrd+F$, where $F\>0$ is an effective $f$-exceptional $\mbZ$-divisor. Therefore by Lemma \ref{lem:projection-formula}, $f_*\mcO_X(D)\cong\mcO_X(f_*D)$.\\
\end{proof}

\begin{lemma}\label{lem:mori-dream-space}
	Let $X$ be a normal surface over an infinite field $k$ of characteristic $p>0$. Let $(X, \Delta\>0)$ be a KLT pair and $-(K_X+\Delta)$ is nef and big $\mbQ$-Cartier $\mbQ$-divisor. Then $X$ is a Mori dream space.
\end{lemma}

\begin{proof}
Since $-(K_X+\Delta)$ is nef and big, there exists an effective divisor $E\>0$ such that $-(K_X+\Delta+\frac{1}{m}E)$ is ample for all $m\gg 0$ (see \cite[Remark 2.4]{Tan18}). Choose $m_0\gg 0$ such that $(X, \Delta+\frac{1}{m_0}E)$ is KLT. Then by replacing $\Delta+\frac{1}{m_0}E$ by $\Delta$ we may assume that $(X, \Delta)$ is KLT and $-(K_X+\Delta)$ is ample.\\
Since $(X, \Delta)$ is KLT, X is $\mbQ$-factorial by \cite[Corollary 4.11]{Tan18}. Let $D$ be a $\mbQ$-divisor on $X$. Then $D-l(K_X+\Delta)$ is ample for some $l\gg 0$ and thus $l'(D-l(K_X+\Delta))$ is very ample for some $l'>0$ sufficiently large and divisible. By the Bertini's theorem over infinite field \cite[Theorem 7 and 7', Page 368 and 376]{Sei50} there exists an irreducible normal curve $A\sim l'(D-l(K_X+\Delta))$ such that the support of $A$ does not contain any component of the support of $\Delta$. Set $\Delta':=\Delta+\frac{1}{ll'}A$. Then the coefficients of $\Delta'$ are in the interval $(0, 1)$ and $\frac{1}{l}D\sim_\mbQ K_X+\Delta'$. Then by \cite[Theorem 1.1]{Tan18} we can run $(K_X+\Delta')$-MMP (since $X$ is $\mbQ$-factorial) and it will terminate with either a minimal model or a Mori fiber space. Since $\frac{1}{l}D\sim_\mbQ K_X+\Delta'$, this is also a $D$-MMP, and hence $X$ is a Mori dream space.
	
\end{proof}

\begin{proposition}\label{pro:running-mmp}
Let $X$ be a regular (resp. normal) del Pezzo surface over an infinite field $k$. Let $(X, \Delta\>0)$ be a KLT pair and $D$ a $\mbZ$-divisor on $X$ such that $0\<A\sim_\mbQ D-(K_X+\Delta)$ is ample. We run a $(\Delta+A)$-MMP and terminate with a birational morphism $f:X\to Y$ 
\begin{equation}
	\xymatrixcolsep{3pc}\xymatrix{f:X=Y_0\ar[r]^-{f_1} & Y_1\ar[r]^-{f_2} & \cdots\ar[r] & \cdots\ar[r]^-{f_n} & Y_n:=Y. }
	\end{equation} 
Then the following conclusions hold.	
\begin{enumerate}
	\item $Y_i$ is a regular (resp. normal) del Pezzo surface (with canonical singularities), for all $i=1, 2,\ldots, n$. In particular, $Y$ is a regular (resp. normal) del Pezzo surface (with canonical singularities).\\
	\item We also have, $f_*\mcO_X(D)\cong\mcO_Y(f_*D)$.
\end{enumerate}
\end{proposition}

\begin{proof}
Note that $X$ is $\mbQ$-factorial (see Subsection \ref{sbs:klt-surface}). Since $-K_X$ is ample, by Lemma \ref{lem:mori-dream-space} $X$, is a Mori dream space. In particular, running a $(\Delta+A)$-MMP makes sense. Since $\Delta+A$ is effective, every $(\Delta+A)$-MMP terminates with a minimal model (not a Mori fiber space). In particular, every $f_i:Y_{i-1}\to Y_i$ is a birational morphism.\\
Set $\Delta_i:=(f_i\circ f_{i-1}\cdots \circ f_1)_*\Delta, A_i:=(f_i\circ f_{i-1}\cdots \circ f_1)_*A, f:=f_n\circ f_{n-1}\circ\cdots\circ f_1, \Delta':=\Delta_n=f_*\Delta$, and $A':=A_n=f_*A$.\\

We will work with the regular del Pezzo case first. Consider $f_1:X\to Y_1$. Since $X$ is regular, it has terminal singularities. Since $f_1$ is a birational contraction of a $K_X$-negative extremal ray (as $-K_X$ is ample), by \cite[Corollary 3.43(3)]{KM98}, $Y_1$ also has terminal singularities. Then by \cite[Theorem 2.29(1)]{Kol13}, $Y_1$ is regular and $K_X=f^*_1K_{Y_1}+aE$, for some $a>0$.\\ 
Let $C$ be a curve on $Y_1$. Then by the projection formula, $K_{Y_1}\cdot C=K_X\cdot f^*_1C<0$ and $K^2_{Y_1}=K_X\cdot f^*_1K_{Y_1}=K_X\cdot(K_X-aE)=(K^2_X-aK_X\cdot E)>0$, since $-K_X$ is ample and $a>0$. Therefore $-K_{Y_1}$ is an ample divisor and $Y_1$ is a regular surface, i.e., $Y_1$ is a regular del Pezzo surface. Then by induction on $i$ it follows that $Y_i$ is a regular del Pezzo surface, for all $i=1, 2,\ldots, n$.\\

Now we will work with the normal del Pezzo case. In this case $X$ has canonical singularities, since $X$ has KLT singularities and $K_X$ is Cartier. Then as in the previous case, by \cite[Corollary 3.43(3)]{KM98} it follows that $Y_1$ also has canonical singularities. Then by \cite[Theorem 2.29(2)]{Kol13} $K_{Y_1}$ is a Cartier divisor. By a similar computation as in the previous case it also follows that $-K_{Y_1}$ is ample. Therefore $Y_1$ is a normal del Pezzo surface with canonical singularities. Then by induction on $i$ it follows that $Y_i$ is a normal del Pezzo surface with canonical singularities, for all $i=1, 2,\ldots, n$.\\ 

Now we prove the second part. Note that, since $D\sim_\mbQ K_X+\Delta+A$ and $-K_X$ is ample, it follows from the previous part that every step of the $(\Delta+A)$-MMP is also a step of the $D$-MMP. Let $E$ be the curve contracted by $f_1:X\to Y_1$. Write $D=f^*_1{f_1}_*D+aE$. Then $-aE$ is $f_1$-nef, since $D\cdot E<0$. Therefore applying the Negativity lemma \cite[Lemma 2.11]{Tan18} we see that $a\>0$. Note that $a\>0$ may not be an integer when $X$ is a normal del Pezzo surface. However, by Corollary \ref{cor:projection-formula} we have, ${f_1}_*\mcO_X(D)\cong\mcO_{Y_1}({f_1}_*D)$.\\

Similarly, for $D_1={f_1}_*D\sim_\mbQ K_{Y_1}+\Delta_1+A_1$ and $f_2:Y_1\to Y_2$ we get, $(f_2\circ f_1)_*\mcO_X(D)\cong {f_2}_*\mcO_{Y_1}({f_1}_*D)\cong \mcO_{Y_2}((f_2\circ f_1)_*D)$. Thus by induction on $i$, we have, $f_*\mcO_X(D)\cong\mcO_Y(f_*D)$, where $f=f_n\circ f_{n-1}\circ\cdots\circ f_1$.\\

\end{proof}

\begin{lemma}\label{lem:nef-big-kv}
	Let $X$ be a normal del Pezzo surface with canonical singularities over an arbitrary field $k$. If $D$ is a nef and big $\mbQ$-Cartier $\mbZ$-divisor on $X$, then $H^1(X, K_X+p^eD)=0$ for all $e>0$ sufficiently large and divisible.
\end{lemma}

\begin{proof}
	Let $m>0$ be the Cartier index of $D$. Write $m=p^al$ such that $a\>0$ and $p$ does not divide $l$. Then $p^a(p^{e-a}-1)D=(p^e-p^a)D$ is Cartier for $e>0$ sufficiently large and divisible. Moreover, $(K_X+p^aD)-K_X=p^aD$ is nef and big, and $D\not\num 0$, since $D$ is big. Therefore by \cite[Theorem 3.8]{Tan18} in $\chr p>0$ or by the Kawamata-Viehweg vanishing theorem in $\chr 0$ we have $H^1(X, \mcO_X(K_X+p^eD))=H^1(X, \mcO_X(K_X+p^aD+(p^e-p^a)D))=0$, for all $e>0$ sufficiently large and divisible.\\
\end{proof}

\begin{proposition}\label{pro:injective-cohomology}
	Let $X$ be a normal del Pezzo surface over an algebraically closed field $k$ of characteristic $p>3$. Let $D\>0$ be an effective $\mbQ$-Cartier nef and big $\mbZ$-divisor on $X$ and $H^1(X, \mcO_X(K_X+p^eD))=0$ for all sufficiently large and divisible $e>0$. Then $H^i(X, \mcO_X(K_X+D))=0$ for all $i>0$.
\end{proposition}

\begin{proof}
	Since $-K_X$ is ample and Cartier and characteristic $p>3$, by \cite[Theorem 1.9]{PW17} $H^1(X, \mcO_X(K_X))=0$. Then by Serre duality (see Subsection \ref{sbs:serre-duality}), $H^1(X, \mcO_X)=H^1(X, \mcO_X(K_X))^*=0$. Therefore by \cite[Lemma 3.2]{CT16jul} there is an injection
	\begin{equation}\label{eqn:injective-cohomology}
		H^1(X, \mcO_X(-D))\injective H^1(X, \mcO_X(-p^eD))
	\end{equation}
	for every positive integer $e>0$.\\
Again, by Serre duality, $H^1(X, \mcO_X(-p^eD))=H^1(X, \mcO_X(K_X+p^eD))^*$. Hence $H^1(X, \mcO_X(-p^eD))=0$, and consequently from  \eqref{eqn:injective-cohomology} it follows that
 \[H^1(X, \mcO_X(K_X+D))=H^1(X, \mcO_X(-D))^*
=0.\]
Since $D$ is effective, we also have $H^2(X, \mcO_X(K_X+D))=H^0(X, \mcO_X(-D))^*=0$.\\		
\end{proof}

\begin{lemma}\label{lem:non-vanishing}
	Let $X$ be a regular projective surface over an arbitrary field $k$ of characteristic $p>3$. Let $L$ be a nef and big $\mbZ$-divisor on $X$. Further assume that one of the following conditions is satisfied:
	\begin{enumerate}
		\item $X$ is a regular del Pezzo surface, or
		\item $-K_X$ is nef and big and $H^1(X, \mcO_X)=0$.
	\end{enumerate} 
	Then $H^0(X, \mcO_X(L))\neq 0$.
\end{lemma}

\begin{proof}
	Since $X$ is a regular surface, by the Riemann-Roch theorem (see \cite[Theorem 2.10]{Tan18}) we have
	\begin{equation}\label{eqn:rr}
		\chi(X, \mcO_X(L))=\frac{1}{2}L\cdot_k(L-K_X)+\chi(X, \mcO_X).
	\end{equation}
Now $\chi(X, \mcO_X(L))=h^0(X, \mcO_X(L))-h^1(X, \mcO_X(L))+h^2(X, \mcO_X(L))$. By Serre duality we have, $h^2(X, \mcO_X(L))=h^0(X, \mcO_X(K_X-L))$. We claim that $h^0(X, \mcO_X(K_X-L))=0$. If $0\neq s\in H^0(X, \mcO_X(K_X-L))$ is a non-zero section, then $0\neq s^m$ is a non-zero section of $H^0(X, \mcO_X(m(K_X-L)))$ for all $m>0$. But $H^0(X, \mcO_X(-m(K_X-L)))\neq 0$ for all $m\gg 0$, since $-(K_X-L)$ is ample (resp. big). Hence, $h^0(X, \mcO_X(K_X-L))=0$. In particular, we have
\begin{equation}\label{eqn:lhs-rr}
	\chi(X, \mcO_X(L))=h^0(X, \mcO_X(L))-h^1(X, \mcO_X(L)).
\end{equation} 
 On the other hand, $L\cdot_k(L-K_X)=L^2+(-K_X\cdot_k L)>0$, since $L$ is nef and big and $-K_X$ is ample (resp. nef). Also, $\chi(X, \mcO_X)=h^0(X, \mcO_X)-h^1(X, \mcO_X)+h^2(X, \mcO_X)$. When $X$ is a regular del Pezzo surface in characteristic $p>3$, by \cite[Theorem 1.9]{PW17} and Serre duality we have $H^1(X, \mcO_X)=0$. Again, by the Serre duality, $H^2(X, \mcO_X)=H^0(X, \mcO_X(K_X))^*=0$, since $-K_X$ is ample (resp. big). Therefore from \eqref{eqn:rr} we get
 \begin{equation}\label{eqn:rhs-rr}
 	\chi(X, \mcO_X(L))=\frac{1}{2}L\cdot_k(L-K_X)+h^0(X, \mcO_X)>0.
 \end{equation} 
Then comparing \eqref{eqn:lhs-rr} and \eqref{eqn:rhs-rr} we get,
\[
	\dim_k H^0(X, \mcO_X(L))>0.
\]	
This concludes the proof.\\

\end{proof}

\begin{corollary}\label{cor:non-vanishing}
	Let $X$ be a normal del Pezzo surface with canonical singularities over an arbitrary field $k$ of characteristic $p>3$. Let $L$ be a nef and big Cartier divisor on $X$. Then $H^0(X, \mcO_X(L))\neq 0$.
\end{corollary}

\begin{proof}
	Since $X$ has canonical singularities, by \cite[Theorem 2.29(2)]{Kol13} either $X$ is regular or $K_X$ is Cartier and there is a resolution of singularities $f:Y\to X$ such that $K_Y=f^*K_X$. If $X$ is regular, then the conclusion follows directly from Lemma \ref{lem:non-vanishing}. So assume that $K_X$ is Cartier and there is a crepant resolution $f:Y\to X$. Then $-K_Y$ and $f^*L$ are both nef and big on $Y$. Thus by the relative Kawamata-Viehweg vanishing theorem \cite[Theorem 10.4]{Kol13}, $R^1f_*\mcO_Y=R^1f_*\mcO_Y(K_Y-K_Y)=0$; in particular, $X$ has rational singularities. Therefore, for any line bundle $\msM$ on $X$, we have $H^i(Y, f^*\msM)=H^i(X, \msM)$, for all $i\>0$. On the other hand, since $X$ is normal and $-K_X$ is ample and Cartier, by \cite[Theorem 1.9]{PW17} and Serre duality (see Subsection \ref{sbs:serre-duality}), $H^1(X, \mcO_X)=H^1(X, \mcO_X(K_X))^*=0$. Hence, $H^1(Y, \mcO_Y)=0$. Then again by Lemma \ref{lem:non-vanishing}, $H^0(X, \mcO_X(L))=H^0(Y, \mcO_Y(f^*L))\neq 0$. 	 
\end{proof}

\section{Main Theorem}

\begin{theorem}\label{thm:main}
Let $(X, \Delta\>0)$ be a projective KLT pair of dimension $2$ over an arbitrary field $k$ of characteristic $p>3$. Let $D$ be a $\mbZ$-divisor on $X$ such that $D-(K_X+\Delta)$ is nef and big. Further assume that one of the following conditions is satisfied:
\begin{enumerate}
	\item $X$ is a regular del Pezzo surface, or
	\item $X$ is a normal del Pezzo surface and $D\>0$ is an effective $\mbZ$-divisor.
\end{enumerate}
 Then $H^i(X, \mcO_X(D))=0$ for all $i>0$.
	\end{theorem}
	
\begin{proof}
Since $X$ is projective over $k$, we have a Stein factorization $X\to \Spec k'\to\Spec k$. Then $k'$ is a finite algebraic extension of $k$. Note that proving the $k$-vector space $H^i(X, \mcO_X(D))$ is zero is same as proving it is zero as a $k'$-vector space. Therefore replacing $k$ by $k'$ we may assume that $H^0(X, \mcO_X)=k$, i.e. $k$ is algebraically closed inside $K(X)$.\\
	
	Next we reduce the problem to the case where $k$ is an infinite field. If $k$ is a finite field, then $k\cong \mbF_{p^e}$ for some integer $e>0$. In particular, $k$ is a perfect field in this case. Note that the algebraic closure of $k$ is: $\bar{k}\cong\overline{\mbF}_p$. Let $X_{\bar{k}}$ be the base-change of $X$ to the algebraic closure $\bar{k}$. Then $X_{\bar{k}}$ is smooth (resp. normal) surface over $\bar{k}$ and $(X_{\bar{k}}, \Delta_{\bar{k}})$ is KLT (see Subsection \ref{sbs:perfect-base-change}). By Lemma \ref{lem:positive-base-change} we also have: $D_{\bar{k}}-(K_{X_{\bar{k}}}+\Delta_{\bar{k}})$ is nef and big, and $-K_{X_{\bar{k}}}$ is an ample Cartier divisor on $X_{\bar{k}}$ (and $D_{\bar{k}}\>0$ is effective in the normal del Pezzo case). Furthermore, by \cite[Proposition 9.3]{Har77} we get 
\begin{equation}\label{eqn:cohomology-base-change}
	H^i(X_{\bar{k}}, \mcO_{X_{\bar{k}}}(D_{\bar{k}}))=H^i(X, \mcO_X(D))\otimes_k\bar{k}.
\end{equation} 	
Therefore $H^i(X, \mcO_X(D))=0$ if and only if $H^i(X_{\bar{k}}, \mcO_{X_{\bar{k}}}(D_{\bar{k}}))=0$. In particular, we may assume that the ground field $k$ is infinite.\\
Note that $X$ is $\mbQ$-factorial (see Subsection \ref{sbs:klt-surface}). By perturbing $\Delta$ we may assume that $D-(K_X+\Delta)$ is ample. Choose $0\<A\sim_\mbQ D-(K_X+\Delta)$. Since $-K_X$ is ample, by Lemma \ref{lem:mori-dream-space} $X$ is a Mori dream space.  We run a $(\Delta+A)$-MMP and terminate with a minimal model $f:X\to Y$ as in the hypothesis of Proposition \ref{pro:running-mmp}. Then from the same proposition it follows that $Y$ is a regular (resp. normal) del Pezzo surface (with canonical singularities). Furthermore, by the relative Kawamata-Viehweg vanishing theorem for surfaces \cite[Theorem 1.3]{Tan18} we have $R^if_*\mcO_X(D)=0$ for all $i>0$. Therefore
\[
	H^i(X, \mcO_X(D))\cong H^i(Y, f_*\mcO_X(D))\cong H^i(Y, \mcO_Y(f_*D)),
\]  
for all $i>0$. The second isomorphism follows from Proposition \ref{pro:running-mmp}.\\
Thus after replacing $X$ by $Y$ we may assume that $\Delta+A$ is nef (and $X$ has canonical singularities). Then $D-K_X\sim_\mbQ \Delta+A$ is a nef and big $\mbZ$-divisor on $X$. If $X$ is a regular del Pezzo surface, then by Lemma \ref{lem:non-vanishing} $H^0(X, \mcO_X(D-K_X))\neq 0$. If $X$ is a normal del Pezzo surface (with canonical singularities), then by Corollary \ref{cor:non-vanishing} $H^0(X, \mcO_X(-K_X))\neq 0$. In this case from our hypothesis we also know that $D$ is effective, therefore we again have $H^0(X, \mcO_X(D-K_X))\neq 0$. By Lemma \ref{lem:nef-big-kv} we also have, $H^1(X, K_X+p^e(D-K_X))=0$ for all $e>0$ sufficiently large and divisible. 

 Now we reduce the problem to the algebraically closed base field. Since $X$ is a normal Gorenstein surface (see Remark \ref{rmk:gorenstein}), $-K_X$ is an ample Cartier divisor, $H^0(X, \mcO_X)=k$ and characteristic $p>3$, by \cite[Theorem 1.5]{PW17}, $X$ is geometrically normal, i.e., the base-change $X_{\bar{k}}$ is normal. Then by Lemma \ref{lem:positive-base-change}, $-K_{X_{\bar{k}}}$ is an ample Cartier divisor, and $D_{\bar{k}}-K_{X_{\bar{k}}}$ is a nef and big $\mbZ$-divisor on $X_{\bar{k}}$. Moreover, using realtions like \eqref{eqn:cohomology-base-change} along with the results proved in the previous paragraph we get that $H^0(X_{\bar{k}}, \mcO_{X_{\bar{k}}}(D_{\bar{k}}-K_{X_{\bar{k}}}))\neq 0$ and $H^1(X_{\bar{k}}, \mcO_{X_{\bar{k}}}(K_{X_{\bar{k}}}+p^e(K_{X_{\bar{k}}}-D_{\bar{k}})))=0$ for all $e>0$ sufficiently large and divisible. Combining all of these with Proposition \ref{pro:injective-cohomology} we have $H^i(X_{\bar{k}}, \mcO_{X_{\bar{k}}}(D_{\bar{k}}))=0$ for all $i>0$. Then again by a similar relation as in \eqref{eqn:cohomology-base-change} we get that, $H^i(X, \mcO_X(D))=0$ for all $i>0$.

\end{proof}

\bibliographystyle{habbrv}
\bibliography{references.bib}

\end{document}